\documentclass[draft,reqno]{amsart}

\usepackage{amsmath}
\usepackage{amsfonts}
\usepackage{amssymb,enumerate}
\usepackage{amsthm}
\usepackage[all]{xy}
\usepackage{rotating}
\usepackage{hyperref}
\usepackage{color}

\theoremstyle{plain}
\newtheorem{lem}{Lemma}[section]

\newtheorem{prop}[lem]{Proposition}
\newtheorem{thm}[lem]{Theorem}

\theoremstyle{definition}
\newtheorem{defn}[lem]{Definition}
\newtheorem{ex}[lem]{Example}
\newtheorem{question}[lem]{Question}

\newtheorem{disc}[lem]{Remark}

\newtheorem{fact}[lem]{Fact}

\newcommand{\cat}[1]{\mathcal{#1}}

\newcommand{\catd}{\cat{D}}

\newcommand{\cata}{\cat{A}}

\newcommand{\catac}{\cat{A}_C}

\newcommand{\catbc}{\cat{B}_C}

\newcommand{\pd}{\operatorname{pd}}	
\newcommand{\gdim}{\mathrm{G}\text{-}\!\dim}	
\newcommand{\gkdim}[1]{\mathrm{G}_{#1}\text{-}\!\dim}	
\newcommand{\gcdim}{\gkdim{C}}

\newcommand{\fd}{\operatorname{fd}}

\newcommand{\edim}{\operatorname{edim}}

\newcommand{\HH}{\operatorname{H}}
\newcommand{\Hom}{\operatorname{Hom}}

\newcommand{\Ker}{\operatorname{Ker}}

\newcommand{\ideal}[1]{\mathfrak{#1}}
\newcommand{\m}{\ideal{m}}
\newcommand{\n}{\ideal{n}}

\newcommand{\comp}[1]{\widehat{#1}}
\newcommand{\ol}{\overline}
\newcommand{\wti}{\widetilde}

\newcommand{\bbz}{\mathbb{Z}}

\newcommand{\xra}{\xrightarrow}

\newcommand{\vf}{\varphi}

\newcommand{\y}{\mathbf{y}}

\newcommand{\x}{\mathbf{x}}

\renewcommand{\geq}{\geqslant}

\renewcommand{\ker}{\Ker}

\newcommand{\Ext}[4][R]{\operatorname{Ext}_{#1}^{#2}(#3,#4)}	
\newcommand{\Rhom}[3][R]{\mathbf{R}\!\operatorname{Hom}_{#1}(#2,#3)}	
\newcommand{\Lotimes}[3][R]{#2\otimes^{\mathbf{L}}_{#1}#3}
\newcommand{\Otimes}[3][R]{#2\otimes_{#1}#3}
\renewcommand{\Hom}[3][R]{\operatorname{Hom}_{#1}(#2,#3)}	
\newcommand{\Tor}[4][R]{\operatorname{Tor}^{#1}_{#2}(#3,#4)}

\newcommand{\catdfb}{\catd^{\text{f}}_{\text{b}}}
\newcommand{\catdb}{\catd_{\text{b}}}

\numberwithin{equation}{lem}

\newcommand{\Hod}{\operatorname{H-dim}}

\begin{document}

\bibliographystyle{amsplain}

\author{Sean K. Sather-Wagstaff}
\address{School of Mathematical and Statistical Sciences,
Clemson University,
O-110 Martin Hall, Box 340975, Clemson, S.C. 29634
USA}
\email{ssather@clemson.edu}
\urladdr{https://ssather.people.clemson.edu/}

\title{Ascent Properties for Test Modules}

\date{\today}

\dedicatory{Dedicated to Roger and Sylvia Wiegand on the occasion of their 150th birthday}


\begin{abstract}
We investigate modules for which vanishing of Tor-modules implies finiteness of homological dimensions (e.g., projective dimension and G-dimension). In particular, we answer a question of O. Celikbas and Sather-Wagstaff about ascent properties of such modules over residually algebraic flat local ring homomorphisms. To accomplish this, we consider ascent and descent properties over local ring homomorphisms of finite flat dimension, and for flat extensions of finite dimensional differential graded algebras. 
\end{abstract}

\maketitle


\section{Introduction} \label{sec190908a}

Throughout, $(R,\m,k)\xra\vf(S,\n,l)$ is a local homomorphism between (commutative noetherian) local rings.

\

A classical result in homological commutative algebra says that for each finitely generated $R$-module $N$, if 
$\Tor ikN=0$ for all $i\gg 0$, then $\pd_R(N)<\infty$.  
In the language of~\cite{celikbas:tgp}, this says that the residue field $k$ is a ``$\pd$-test $R$-module'':
a finitely generated $R$-module 
$M$ is a \emph{$\pd$-test module} over $R$ provided that for each finitely generated $R$-module $N$,
if $\Tor iMN=0$ for all $i\gg 0$, then $\pd_R(N)<\infty$.  

A key point of our work with O. Celikbas~\cite{celikbas:tgp} is to prove that if $M$ is $\pd$-test over $R$, 
then the $\m$-adic completion $\comp M$ is $\pd$-test over $\comp R$. (The converse is relatively straightforward to prove.)
This is accomplished by considering the more general notion of $\pd$-test \emph{complexes}.
This level of generality allows us to overcome certain technical difficulties involved in considering this ascent problem. 
However, it did not answer the following.

\begin{question}\label{q191102a}
If $\vf$ is flat with regular closed fibre $S/\m S$ and
$M$ is a $\pd$-test $R$-module, must $\Otimes SM$ be $\pd$-test over $S$?
\end{question}

It is straightforward to show that the assumption on the closed fibre in this question is necessary: if $R=k$ and $S=k[X]/\langle X^2\rangle$,
then $k$ is $\pd$-test over $k$, but $S=\Otimes[k]Sk$ is not $\pd$-test over $S$ since $S$ is not regular.

In the current paper, we answer Question~\ref{q191102a} in the case where the induced residue field extension $k\to l$ is algebraic;
see Theorem~\ref{thm140915a}. Moreover, the proof shows that the general case reduces to the special case of a purely transcendental field extension with finite transcendence degree.

\

\noindent\textbf{Theorem~\ref{thm140915a} (module case).}
\emph{Assume that $\vf$ is flat with regular closed fibre, and let $M$ be a finitely generated $R$-module.
Assume the induced field extension $k \to l$ is algebraic.
Then $M$ is a  $\pd$-test module over $R$ if and only if $\Otimes SM$ is a $\pd$-test module over $S$.}

\

Our proof of Theorem~\ref{thm140915a} uses two more significant expansions of our original context. 
First, we consider the more general setting of local ring homomorphisms of finite flat dimension. 
This perspective was pioneered by Avramov and Foxby in a sequence of 
articles~\cite{avramov:glh,avramov:lgh,avramov:glp,avramov:cmporh,avramov:dalp,avramov:solh,avramov:bsolrhoffd} 
where they 
used this notion to answer open questions about flat local ring homomorphisms.
Section~\ref{sec150429a} consists of foundational material about test modules and homomorphisms of finite flat dimension.

Second, we expand our perspective from modules and complexes over local rings to the setting of 
differential graded (DG) modules over commutative DG algebras.
This point of view was developed by Avramov and his 
collaborators; see, e.g.,~\cite{avramov:ifr,avramov:cslrec3, avramov:glh, avramov:lgh,avramov:dgha,avramov:dalp,avramov:bsolrhoffd,avramov:tlg,avramov:holh,avramov:htecdga, avramov:phcnr,avramov:psmlrsec}.
It provides a construction whereby if one can solve a homological commutative algebra problem for finite dimensional algebras over a field, then one can sometimes 
solve the general problem by passing to an associated finite dimensional DG algebra over a field where one can solve a related problem.
We have used this approach in several projects~\cite{altmann:csc,beck:sgidgca,christensen:dvke,nasseh:ldgm,nasseh:egdgm,nasseh:gart}.
For the current setting, we develop some technology around $\pd$-test DG modules over finite dimensional DG algebras in Section~\ref{sec190908c}
and then use it to prove Theorem~\ref{thm140915a}.

To be clear, we consider these generalizations because we do not know how to prove Theorem~\ref{thm140915a} without them.
In addition, our techniques allow us to answer similar ascent questions for  objects that test for finiteness of G-dimensions.
This is the subject of Section~\ref{sec190913a}.
In addition, we include some necessary background material in Section~\ref{sec140831}.

\section{Derived Categories and Semidualizing Complexes}
\label{sec140831}

Throughout this paper we work in the derived category $\catd(R)$ whose objects are the chain complexes of $R$-modules, i.e., the
$R$-complexes 
$$X=\qquad\cdots\xra{\partial^X_{i+1}} X_i\xra{\partial^X_{i}} X_{i-1}\xra{\partial^X_{i-1}}\cdots.$$
References for this include~\cite{christensen:dcmca,hartshorne:rad, verdier:cd, verdier:1}. 
We denote by $\Rhom XY$ and $\Lotimes XY$  the derived Hom-complex and derived tensor product of two $R$-complexes $X$ and $Y$.
Isomorphisms in $\catd(R)$ are identified by the symbol $\simeq$.
A complex $X\in\catd(R)$ has \emph{finite projective dimension}, signified $\pd_R(X)<\infty$, if it is isomorphic to a bounded complex of projective $R$-modules.
Finiteness of \emph{flat dimension} and \emph{injective dimension} are signified and defined similarly. 
The ring homomorphism $R\xra\vf S$ has \emph{finite flat dimension} provided that $\fd_R(S)<\infty$.

The subcategory of $\catd(R)$ consisting of homologically bounded $R$-complexes
(i.e., complexes $X$ such that $\HH_i(X)=0$ for $|i|\gg 0$) is denoted $\catdb(R)$.
The subcategory of $\catd(R)$ consisting of homologically finite $R$-complexes
(i.e., complexes $X$ such that $\HH(X):=\oplus_{i\in\bbz}\HH_i(X)$ is finitely generated) is denoted $\catdfb(R)$.

A complex $C\in\catdfb(R)$ is \emph{semidualzing} if the natural 
homothety morphism $R\to\Rhom CC$ in $\catd(R)$ is an isomorphism.
Consequently, an $R$-module is semidualizing if and only if $\Hom CC\cong R$ and $\Ext iCC=0$ for $i\geq 1$.
In particular, $R$ is a semidualizing $R$-module.
A \emph{dualizing $R$-complex} is a semidualizing $R$-complex of finite injective dimension.
Dualizing complexes were introduced in~\cite{hartshorne:rad}.
The more general semidualizing complexes were defined in~\cite{christensen:scatac},
based on special cases in~\cite{avramov:rhafgd, foxby:gmarm, golod:gdagpi, vasconcelos:dtmc}.

\begin{fact}\label{ch150423c1}
If $R$ is a homomorphic image of a local Gorenstein ring $Q$, then $R$ has a dualizing complex, by~\cite[\S V.10]{hartshorne:rad}.
(The converse holds by~\cite[Corollary~6.2]{kawasaki:mns}.)
In particular, the Cohen Structure Theorem shows that the completion $\comp R$ has a dualizing complex.
When $R$ has a dualizing complex $D$, and $C$ is a semidualizing $R$-complex, the dual $\Rhom CD$ is also semidualizing over $R$, by~\cite[(2.12)~Corollary]{christensen:scatac}.
\end{fact}

\begin{fact}\label{ch150423c2}
Assume that $\vf$ has finite flat dimension,
and let $C$ be a semidualizing $R$-complex.
Then the $S$-complex $\Lotimes SC$ is semidualizing, by~\cite[(5.10)~Theorem]{christensen:scatac}.
If $\vf$ is flat and the closed fibre $S/\m S$ is Gorenstein and $R$ has a dualizing complex $D^R$, then
$D^S:=\Lotimes S{D^R}$ is dualizing for $S$ by~\cite[(5.1)~Theorem]{avramov:lgh}.
\end{fact}

In order to describe another class of ring homomorphisms where dualizing complexes base change to dualizing complexes, 
we need the following. 
A \emph{Cohen factorization} of $\vf$ is a 
decomposition of $\vf$ as a composition of local ring homomorphisms
$R\xra{\dot\vf}R'\xra{\vf'}S$ such that $\dot\vf$ is flat with regular closed fibre, $R'$ is complete, and $\vf'$ is surjective. 
Such a decomposition exists if and only if $S$ is complete by~\cite[(1.1)~Theorem]{avramov:solh}.
In particular, the \emph{semicompletion} of $\vf$, which is the composition $\grave\vf$ of $\vf$ with the natural map
from $S$ to its completion $\comp S$, has a Cohen factorization.

One idea with Cohen factorizations is that it allows one to transfer questions and properties for $S$ as an $R$-module
to related questions and properties for $\comp S$ has an $R'$-module. This is valuable because $\comp S$ is finitely generated over $R'$
while $S$ often is not finitely generated over $R$. An example of this transfer is found in~\cite[(3.2)~Lemma]{avramov:solh} which implies that
$\fd(\vf)<\infty$ if and only if $\pd_{R'}(\comp S)<\infty$. 

Recall that an ideal $I$ of $R$ is \emph{Gorenstein} if $g=\pd_R(R/I)<\infty$ satisfies
$\Ext{i}{R/I}R=0$ for all $i\neq g$ and $\Ext g{R/I}R\cong R/I$. 
Our  local ring homomorphism $\vf$ is \emph{Gorenstein} provided that
in some (equivalently, every) Cohen factorization $R\xra{\dot\vf}R'\xra{\vf'}\comp S$ of $\grave\vf$ the ideal $\ker(\vf')$ of $R'$ is Gorenstein;
see~\cite{avramov:glh,avramov:lgh}.

\begin{fact}\label{para190908b}
If $\vf$ is flat, then $\vf$ is Gorenstein if and only if its closed fibre is Gorenstein.
In general, if $\vf$ is Gorenstein and $R$ has a dualizing complex $D^R$, then
$D^S:=\Lotimes S{D^R}$ is dualizing for $S$ by~\cite[(5.1)~Theorem]{avramov:lgh}.
\end{fact}

As the name suggests, semidualizing complexes are built for duality.
Let $C$ be a semidualizing $R$-complex and $X\in\catdfb(R)$. 
We say that $X$ is \emph{derived $C$-reflexive} and write  $\gcdim_R(X)<\infty$ when
$\Rhom XC\in\catdb(R)$ and the natural morphism $X\to\Rhom{\Rhom XC}C$ in $\catd(R)$ is an isomorphism.
In the case $C=R$, we write $\gdim_R(X)<\infty$ instead of $\gkdim R_R(X)<\infty$.
In the special case where $C$ is dualizing, this is from~\cite{hartshorne:rad},
while the case $C=R$ comes from~\cite{auslander:adgeteac,auslander:smt,yassemi:gd}.
The general situation is in~\cite{christensen:scatac}.

\begin{fact}\label{ch150423a1}
A semidualizing complex $C$ is dualizing if and only if every $R$-complex in $\catdfb(R)$ is derived $C$-reflexive, by~\cite[(8.4)~Proposition]{christensen:scatac}.
In particular, $R$ is Gorenstein if and only if every $R$-complex $X\in\catdfb(R)$ has $\gdim_R(X)<\infty$.
\end{fact}

\begin{fact}\label{ch150423a3}
Assume that $\fd(\vf)<\infty$. 
If $X\in\catdfb(R)$, then
$\gcdim_R(X)<\infty$ if and only if $\gkdim{\Lotimes SC}_S(\Lotimes SX)<\infty$, by~\cite[(5.10)~Theorem]{christensen:scatac};  Fact~\ref{ch150423c2}
shows that this is reasonable.
\end{fact}

Auslander and Bass classes, defined next, arrived in special cases in~\cite{avramov:rhafgd,foxby:gmarm},
again with the general case described in~\cite{christensen:scatac}.
Let $C$ be a semidualizing $R$-complex.
The \emph{Auslander class}  $\catac(R)$ consists of the $R$-complexes $X\in\catdb(R)$
such that $\Lotimes CX\in\catdb(R)$ and the natural  morphism $\gamma^C_X\colon X\to\Rhom C{\Lotimes CX}$ in $\catd(R)$ is an isomorphism.
The \emph{Bass class} $\catbc(R)$ consists of all the $R$-complexes $X\in\catdb(R)$
such that $\Rhom CX\in\catdb(R)$ and such that the natural evaluation morphism $\xi^C_X\colon\Lotimes C{\Rhom CX}\to X$ in $\catd(R)$ is an isomorphism.

\begin{fact}\label{ch150423b1}
When $R$ has a dualizing complex $D$, given an $R$-complex $X\in\catdfb(R)$, one has
$\gcdim_R(X)<\infty$ if and only if $X\in\cata_{\Rhom CD}(R)$, by~\cite[(4.7)~Theorem]{christensen:scatac}; this uses Facts~\ref{ch150423c1} and~\ref{ch150423a1}, which imply that
$\Rhom CD$ is semidualizing and $C\simeq\Rhom{\Rhom CD}D$.
\end{fact}

\begin{fact}\label{ch150423b2}
Assume that $\fd(\vf)<\infty$, and let $X\in\catdb(S)$. Then
$X\in\catac(R)$ if and only if $X\in\cata_{\Lotimes SC}(S)$, by~\cite[(5.3.a)~Proposition and~(5.10)~Theorem]{christensen:scatac}.
\end{fact}

The proof of the following  lemma is similar to that of~\cite[7.3~Corollary]{iyengar:golh},
but it is different enough that we include a proof.

\begin{lem}\label{lem150421a}
Assume that $\vf$ is Gorenstein, and let $X\in\catdfb(S)$ be such that each homology module $\HH_i(X)$ is finitely generated over $R$.
\begin{enumerate}[\rm(a)]
\item \label{lem150421a1}
One has $\gcdim_R(X)<\infty$ if and only if $\gkdim{\Lotimes SC}_S(X)<\infty$.
\item \label{lem150421a2}
One has $\gdim_R(X)<\infty$ if and only if $\gdim_S(X)<\infty$.
\end{enumerate}
\end{lem}

\begin{proof}
\eqref{lem150421a1}
Case 1: $R$ has a dualizing compelx $D^R$. From Fact~\ref{para190908b}, we have $D^S\simeq\Lotimes{S}{D^R}$ is dualizing for $S$.
And Facts~\ref{ch150423b1}--\ref{ch150423b2} say that $\gcdim_R(X)<\infty$ if and only if $X\in\cata_{\Rhom C{D^R}}(R)$
if and only if $X\in\cata_{\Lotimes S{\Rhom C{D^R}}}(S)$.
Standard isomorphisms give 
$$\Lotimes S{\Rhom C{D^R}}\simeq \Rhom[S]{\Lotimes SC}{\Lotimes S{D^R}}\simeq \Rhom[S]{\Lotimes SC}{D^S}
$$
so $\gcdim_R(X)<\infty$ if and only if $X\in\cata_{\Rhom[S]{\Lotimes SC}{D^S}}(S)$;
Facts~\ref{ch150423a1} and~\ref{ch150423b1} show that these are equivalent to $\gkdim{\Lotimes SC}_S(X)<\infty$, as desired.

Case 2: we have $\m\HH(X)=0$.
In this case, let $\comp R$ and $\wti S$ denote the $\m$-adic completions of $R$ and $S$, respectively, and let $\wti\vf\colon\comp R\to\wti S$
be the $\m$-adic completion of $\vf$, which is Gorenstein. 
By Fact~\ref{ch150423a3} we have 
$\gcdim_R(X)<\infty$ if and only if $\gkdim{\Lotimes {\comp R}C}_{\comp R}(\Lotimes {\comp R}X)<\infty$.
As in the proof of~\cite[7.1~Theorem]{iyengar:golh} we have an $\comp R$-isomorphism
$\Lotimes {\comp R}X\simeq \Lotimes[S]{\wti S}X$, so   
$\gkdim{\Lotimes {\comp R}C}_{\comp R}(\Lotimes {\comp R}X)<\infty$
if and only if $\gkdim{\Lotimes[\comp R]{\wti S}{(\Lotimes {\comp R}C})}_{\wti S}(\Lotimes[S]{\wti S}X)<\infty$ by Case~1,
i.e.,  if and only if $\gkdim{\Lotimes{\wti S}{C}}_{\wti S}(\Lotimes[S]{\wti S}X)<\infty$.
Again by Fact~\ref{ch150423a3} we have 
$\gkdim{\Lotimes SC}_S(X)<\infty$ if and only if $\gkdim{\Lotimes[S]{\wti S}{(\Lotimes SC})}_{\wti S}(\Lotimes[S]{\wti S}X)<\infty$
if and only if $\gkdim{\Lotimes{\wti S}{C}}_{\wti S}(\Lotimes[S]{\wti S}X)<\infty$, as desired.

Case 3: the general case.
Let $K^R$ be the Koszul complex over $R$ on a finite generating set for $\m$.
Then 
$\gcdim_R(X)<\infty$ if and only if $\gcdim_R(\Lotimes{K^R}X)<\infty$ by~\cite[Theorem~4.4]{frankild:rrhffd},
that is, if and only if $\gcdim_R(\Lotimes[S]{(\Lotimes{K^R}S)}X)<\infty$.
Case~2 says that this is equivalent to $\gkdim{\Lotimes SC}_R(\Lotimes[S]{(\Lotimes{K^R}S)}X)<\infty$,
which is equivalent to $\gkdim{\Lotimes SC}_S(X)<\infty$ by another application of~\cite[Theorem~4.4]{frankild:rrhffd}.

\eqref{lem150421a2} This is the special case $C=R$ of part~\eqref{lem150421a1}.
\end{proof}

\section{Test Complexes and Ring  Homomorphisms of Finite Flat Dimension}
\label{sec150429a}

In this section, let $C$ be a semidualizing $R$-complex.

\

Let $M\in\catdfb(R)$, and let $\Hod$ denote either $\pd$ or $\gcdim$.
Then $M$ is an \emph{$\Hod$-test complex} over $R$ if the following condition holds for all  $N\in\catdfb(R)$:
If $\Tor iMN=0$ for all $i\gg 0$, i.e., if $\Lotimes MN\in\catdb(R)$, then $\Hod_R(N)<\infty$.  
See~\cite[Section~3]{celikbas:tgp} for examples and basic properties of these objects.

The following two results are proved  like~\cite[Theorems~3.2 and~3.4]{celikbas:tgp}, using 
Lemma~\ref{lem150421a} and Fact~\ref{ch150423a3}.

\begin{prop}\label{prop140825a}
Assume that $\fd(\vf)<\infty$, and let $M\in\catdfb(R)$.
\begin{enumerate}[\rm(a)]
\item\label{prop140825a2}
If $\Lotimes SM$ is $\gkdim{\Lotimes SC}$-test  over $S$, then $M$ is   $\gcdim$-test  over $R$.
\item\label{prop140825a2'}
If $\Lotimes SM$ is $\gdim$-test  over $S$, then $M$ is   $\gdim$-test  over $R$.
\item\label{prop140825a1}
If $\Lotimes SM$ is  $\pd$-test  over $S$, then $M$ is $\pd$-test  over $R$.
\end{enumerate}
\end{prop}

\begin{prop}\label{thm140825ax}
Assume that $\vf$ is Gorenstein, and let $M\in\catdfb(R)$.
Assume that the induced residue field extension $k\to l$ is finite.
\begin{enumerate}[\rm(a)]
\item \label{thm140825ax1}
$M$ is $\gcdim$-test over $R$ if and only if $\Lotimes SM$ is $\gkdim{\Lotimes SC}$-test  over $S$.
\item \label{thm140825ax2}
$M$ is  $\gdim$-test over $R$ if and only if $\Lotimes SM$ is   $\gdim$-test  over $S$.
\end{enumerate}
\end{prop}

\begin{disc}\label{disc190914a}
Assume that $\vf$ is flat with regular closed fibre.
Let $\y=y_1,\ldots,y_n\in\n$ be a sequence that forms a regular system of parameters for the regular local ring $S/\m S$.
Then $\y$ is $S$-regular and that the composition $\tau\vf\colon R\to \ol S=S/(\y)$ of $\vf$ with the quotient map
$\tau\colon S\to \ol S$ is flat by, e.g.,~\cite[Corollary to Theorem~22.5]{matsumura:crt}. By construction, the closed fibre of $\tau\vf$ is $S/(\m,\y)=l$. 
\end{disc}

Here is one of our main results. One point is that one can prove better results about flat local maps by widening the context to finite flat dimension.
Note that the case where $\m S=\n$ is covered by~\cite[Theorem~3.5]{celikbas:tgp}.
As we note in~\cite[Example~3.6]{celikbas:tgp}, if $S/\m S$ is  not regular, then the ascent portion of this result fails.

\begin{thm}\label{thm140825a}
Assume that $\vf$ is flat with regular closed fibre, and let $M\in\catdfb(R)$.
Assume the induced field extension $k \to l$ is finite.
Then $M$ is a  $\pd$-test complex over $R$ if and only if $\Lotimes SM$ is a $\pd$-test complex over $S$.
\end{thm}

\begin{proof}
One implication is covered by Proposition~\ref{prop140825a}\eqref{prop140825a1}.
For the reverse implication, assume that $M$ is a  $\pd$-test complex over $R$.
Follow the notation from Remark~\ref{disc190914a}.
From~\cite[Theorem~3.5]{celikbas:tgp}, it follows that $\Lotimes{\ol S}M=\Lotimes[S]{\ol S}{(\Lotimes SM)}$ is $\pd$-test over $\ol S$,
so Proposition~\ref{prop140825a}\eqref{prop140825a1} implies that $\Lotimes SM$ is  $\pd$-test  over $S$.
\end{proof}

\section{PD-Test Results} \label{sec190908c}

As we discuss in the introduction, the point of this section is to prove Theorem~\ref{thm140915a},
and the proof relies heavily on a version of pd-test objects for finite dimensional DG algebras. 
One may consult any of the following for background on DG algebras and derived categories of
DG modules~\cite{avramov:ifr,avramov:dgha,avramov:htecdga,beck:sgidgca,felix:rht}. 

\begin{defn} \label{defn190913a} 
We say that a (positively graded commutative) DG algebra $A$ is \emph{weakly local} if
$\HH_0(A)$ is local and noetherian and $\HH(A)$ is finitely generated as a module over $\HH_0(A)$.
In particular, $\HH(A)$ is bounded.
In this situation, let $\m_A$ be the augmentation ideal of $A$ corresponding to the maximal ideal $\m_{\HH_0(A)}\subset\HH_0(A)$, and
set $k=A/\m_A$. We sometimes summarize this by writing that $(A,\m_A,k)$ is a weakly local DG algebra.
\end{defn}

\begin{disc}\label{disc190913b}
If $k$ is a field and $A$ is a finite dimensional DG $k$-algebra such that $A_0=k$ and
$\HH_0(A)\neq 0$, then $A$ is weakly local with
$\m_A=A_+$ and $\partial^A_1=0$ and $A/\m_A\cong k\cong\HH_0(A)$.
\end{disc}

\begin{fact}\label{fact190913a}
Let $(A,\m_A,k)$ be a weakly local DG algebra.
For a given $N\in\catdfb(A)$, the following conditions are equivalent.
\begin{enumerate}[\rm(i)]
\item \label{fact190913a1}
$N$ has a semi-free resolution $F\simeq N$ over $A$ with a finite semi-basis.
\item \label{fact190913a2}
For all $L\in\catdb(A)$, one has $\Lotimes[A]NL\in\catdb(A)$.
\item \label{fact190913a3}
One has $\Lotimes[A]Nk\in\catdb(A)$.
\end{enumerate}
For the implication~\eqref{fact190913a1}$\implies$\eqref{fact190913a2}, replace $A$ with a bounded truncation and use $F$ to compute $\Lotimes[A]NL$.
The implication~\eqref{fact190913a3}$\implies$\eqref{fact190913a1} is in~\cite[Proposition~B.9]{avramov:htecdga}.
Note that in this situation, the resolution $F$ will be homologically bounded but not necessarily bounded, unless $A$ is bounded.
Furthermore, for the implication~\eqref{fact190913a1}$\implies$\eqref{fact190913a2}, it is crucial that $A$ be homologically bounded.
\end{fact}

Here is the DG version of $\pd$-test objects we use to prove Theorem~\ref{thm140915a}.

\begin{defn} \label{defn190913b} 
Let $A$ be a weakly local DG algebra.
We say that $N\in\catdfb(A)$ is \emph{perfect} over $A$ if $M$ satisfies the equivalent conditions of Fact~\ref{fact190913a}.
Then $M\in\catdfb(A)$ is \emph{$\pd$-test} over $A$ if the following condition holds for all  $N\in\catdfb(A)$:
If $\Lotimes[A] MN\in\catdb(A)$, then $N$ is perfect.
\end{defn}

\begin{ex}\label{disc190913a}
If $(A,\m_A,k)$ is a weakly local DG algebra, then Fact~\ref{fact190913a} implies that $k$ is $\pd$-test over $A$.
\end{ex}

Our proof of Theorem~\ref{thm140915a} follows almost directly from the following result via a construction of Avramov.

\begin{thm}\label{thm190913a}
Let $A$ be a finite-dimensional DG $k$-algebra with $A_0=k$ and $\HH_0(A)\neq 0$.
Let $k\to l$ be a  field extension, set $B=\Otimes[k]lA$, and consider the natural morphism of DG algebras $A\to B$.
Let $M\in\catdfb(A)$ be given.
\begin{enumerate}[\rm(a)]
\item \label{thm190913a1}
If 
$\Lotimes[A] BM$ is $\pd$-test  over $B$, then 
$M$ is  $\pd$-test  over $A$.
\item \label{thm190913a2}
The converse of part~\eqref{thm190913a1} holds if the extension $k\to l$ is algebraic.
\end{enumerate}
\end{thm}

\begin{proof}
\eqref{thm190913a1}
We argue as in the proof of~\cite[Theorem~3.2]{celikbas:tgp}. 
Assume that $\Lotimes[A]BM$ is $\pd$-test  over $B$.
To show that $M$ is $\pd$-test  over $A$,  let $N\in\catdfb(A)$ be such that
$\Lotimes[A] MN\in\catdb(A)$.
As $N$ is homologically finite over $A$, which is finite dimensional over $k$,
we have
$$\Lotimes[A] BN\simeq\Lotimes[A] {(\Otimes[k]lA)}N\simeq\Otimes[k]lN\in\catdfb(l)
$$
so $\Lotimes[A] BN\in\catdfb(B)$,
and similarly $\Lotimes[A] BM,\Lotimes[A] B{(\Lotimes[A] MN)}\in\catdfb(B)$.
Moreover, we have the following isomorphisms
in $\catd(B)$:
$$\Lotimes[B]{(\Lotimes[A] BM)}{(\Lotimes[A] BN)}
\simeq\Lotimes[A]{(\Lotimes[A] BM)}N
\simeq\Lotimes[A] B{(\Lotimes[A] MN)}.$$
As $\Lotimes[A] BM$ is $\pd$-test over $B$,  the DG $B$-module $\Lotimes[A]BN$ is perfect.
It is straightforward to show (arguing as above or using a minimal semi-free resolution of $N$ over $A$, with Fact~\ref{fact190913a})
that $N$ is perfect over $A$, as desired.

\eqref{thm190913a2}
Assume now that $M$ is  $\pd$-test  over $A$.

Case 1: $k\to l$ is finite.
To show that $\Lotimes[A] {B}M$ is $\pd$-test over $B$, 
let $N\in\catdfb(B)$ be such that
$\Lotimes[B] {(\Lotimes[A] {B}M)}N\in\catdb(B)$.
Since $k\to l$ is finite, we have $B\in\catdfb(A)$, so 
$N\in\catdfb(A)$.
Moreover, 
$\Lotimes[A]MN\simeq\Lotimes[B] {(\Lotimes[A] {B}M)}N\in\catdb(A)$.
As $M$ is $\pd$-test  over $A$, the DG module $N$ is perfect over $A$.
To conclude that $N$ is perfect over $B$,
it suffices to show that $\Lotimes[B]Nl\in\catdb(B)$, equivalently, that $\Lotimes[B]Nl\in\catdb(A)$.
By assumption, we have $\Lotimes[A]Nk\in\catdb(A)$.
By construction, we have 
$$\Otimes[A]Bk\cong\Otimes[A]{(\Otimes[k]lA)}k\cong\Otimes[k]lk\simeq l
$$
so
$$\Lotimes[B]Nl\simeq\Lotimes[B]N{(\Otimes[A]Bk)}\simeq\Lotimes[A]Nk\in\catdb(A)
$$
as desired.

Case 2: $k\to l$ is algebraic.
To show that $\Lotimes[A] {B}M$ is $\pd$-test over $B$, 
let $N\in\catdfb(B)$ be such that
$\Lotimes[B] {(\Lotimes[A] {B}M)}N\in\catdb(B)$.
Truncate a degreewise finite $B$-semifree resolution of $N$ if necessary to assume without loss of generality that
$N$ is finite dimensional over $l$.
It follows that the differential and scalar multiplication on $N$ are represented by matrices consisting of finitely many elements of $l$.
Let $k'$ be the intermediate field extension $k\to k'\to l$ generated over $k$ by this finite set of elements. 
Since $l$ is algebraic over $k$, the extension $k\to k'$ is finite.
Set $A'=\Otimes[k]{k'}A$.
Since $k'$ contains the elements representing the differential and scalar multiplication on $N$, there is a DG $A'$-module $L$ 
that is bounded and degreewise finite over $k'$ such that $N\cong \Otimes[A']BL$. 

Using these constructions, we compute:
\begin{align*}
\Lotimes[B] {(\Lotimes[A] {B}M)}N
&\simeq\Lotimes[B] {(\Lotimes[A'] {B}{(\Lotimes[A]{A'}M)})}{(\Otimes[A']BL)}\\
&\simeq\Lotimes[A']{B}{(\Lotimes[A']{(\Lotimes[A]{A'}M)}{L})}\\
&\simeq\Lotimes[A']{(\Lotimes[k']{l}{A'})}{(\Lotimes[A']{(\Lotimes[A]{A'}M)}{L})}\\
&\simeq\Lotimes[k']{l}{(\Lotimes[A']{(\Lotimes[A]{A'}M)}{L})}
\end{align*}
Thus, we have
$\Lotimes[k']{l}{(\Lotimes[A']{(\Lotimes[A]{A'}M)}{L})}\simeq\Lotimes[B] {(\Lotimes[A] {B}M)}N\in\catdb(B)$.
Since $k'\to l$ is faithfully flat, it follows that $\Lotimes[A']{(\Lotimes[A]{A'}M)}{L}$ is homologically bounded as well.
Case 1 implies that $\Lotimes[A]{A'}M$ is $\pd$-test over $A'$,
so the homological boundedness of $\Lotimes[A']{(\Lotimes[A]{A'}M)}{L}$ implies that $L$ is perfect over $A'$.
It follows readily that $N\simeq\Lotimes[A']BL$ is perfect over $B$, as desired.
\end{proof}

The next result uses the exterior DG algebra structure on the Koszul complex.

\begin{lem}\label{lem190915a}
Let $K=K^R(\x)$ be the Koszul complex on $\x=x_1,\ldots,x_n\in\m$, and let $M\in\catdfb(R)$.
\begin{enumerate}[\rm(a)]
\item \label{lem190915a1}
If $\Lotimes KM$ is $\pd$-test over $K$, then $M$ is $\pd$-test over $R$.
\item \label{lem190915a2}
The converse of part~\eqref{lem190915a1} holds when $\x$ is part of a minimal generating sequence for $\m$.
\end{enumerate}
\end{lem}

\begin{proof}
\eqref{lem190915a1}
Argue as in the descent result~\cite[Theorem~3.2]{celikbas:tgp}.

\eqref{lem190915a2}
We argue as in the proof of Theorem~\ref{thm190913a}\eqref{thm190913a2}, with a twist.
Assume  that $M$ is a  $\pd$-test complex over $R$.
To show that $\Lotimes{K}M$ is $\pd$-test over $K$, 
let $N\in\catdfb(K)$ be such that
$\Lotimes[K] {(\Lotimes{K}M)}N\in\catdb(K)$.
Since $K\in\catdfb(R)$, we have 
$N\in\catdfb(R)$.
Moreover, 
$\Lotimes MN\simeq\Lotimes[K] {(\Lotimes{K}M)}N\in\catdb(R)$.
As $M$ is $\pd$-test  over $R$, the DG module $N$ is perfect over $R$, so $\Lotimes kN\in\catdb(R)$.
To conclude that $N$ is perfect over $K$, we compute:
\begin{align*}
\Lotimes kN
&\simeq\Lotimes[K]{(\Lotimes kK)}{N}
\simeq\Lotimes[K]{\HH(\Lotimes kK)}N
\simeq\Lotimes[k]{\HH(\Lotimes kK)}{(\Lotimes[K]kN)}
\end{align*}
The second isomorphism here is from~\cite[Theorem~9.1]{MR2592508}; this is where we use the fact that $\x$ is part of a minimal
generating sequence for $\m$.
Since $\Lotimes kN$ is homologically bounded, 
the K\"unneth formula implies that $\Lotimes[K]kN\in\catdb(K)$, so $N$ is perfect.
\end{proof}

Here is our main result about $\pd$-text complexes.

\begin{thm}\label{thm140915a}
Assume that $\vf$ is flat with regular closed fibre, and let $M\in\catdfb(R)$.
Assume the induced field extension $k \to l$ is algebraic.
Then $M$ is a  $\pd$-test complex over $R$ if and only if $\Lotimes SM$ is a $\pd$-test complex over $S$.
\end{thm}

\begin{proof}
As with Theorem~\ref{thm140825a}, one implication is covered by Proposition~\ref{prop140825a}\eqref{prop140825a1}.
For the reverse implication, assume that $M$ is a  $\pd$-test complex over $R$.

Case 1: $R$ and $S$ are complete and $\m S=\n$. 
Let $\tau\colon Q\to R$ be a minimal Cohen presentation of $R$, so the map is a surjective ring homomorphism where
$Q$ is a complete regular local ring with the same embedding dimension as $R$.
Apply~\cite[(1.6)~Theorem]{avramov:solh} to find a commutative diagram of local homomorphisms of complete local rings
$$\xymatrix{Q\ar[r]^-\alpha\ar[d]_-\tau&Q'\ar[d]^-{\tau'}\\
R\ar[r]^-\vf&S}$$
such that $\alpha$ is flat with regular closed fibre, $\tau'$ is surjective, and the induced map
$\Otimes[Q]R{Q'}\to S$ is an isomorphism. 
Moreover, the proof of \textit{loc.\ cit.} shows that $\edim(Q'/\m_Q Q')=\edim(S/\m S)=0$, so the closed fibre of $\alpha$ is $Q'/\m_Q Q'\cong l$.

Let $F\xra\simeq R$ be a bounded degreewise finite semi-free DG algebra resolution of $R$ over $Q$.
The isomorphism $\Otimes[Q]R{Q'}\cong S$ from the previous paragraph implies that $F'=\Otimes[Q]F{Q'}\xra\simeq S$ is a
bounded degreewise finite semi-free DG algebra resolution of $S$ over $Q'$ since $\alpha$ is flat.
Let $\x$ be a minimal generating sequence for $\m$. Our assumptions imply that $\vf(\x)$ is a minimal generating sequence for $\n$.
Set $K^R=K^R(\x)$ and $K^S=K^S(\vf(\x))$ and consider the following standard commutative diagram of morphisms of DG algebras
where $K^Q$ is the Koszul complex over $Q$ on a lift $\y$ of the sequence $\x$, and $K^{Q'}=K^{Q'}(\alpha(\y))\cong\Otimes[Q]{Q'}{K^Q}$.
\begin{equation}\begin{split}\label{eq191118a}\xymatrix{
R\ar[r]\ar[d]_-\vf&K^R\ar[d]&\Otimes[Q]{R}{K^Q}\ar[d]\ar[l]_-\cong&\Otimes[Q]{F}{K^Q}\ar[d]\ar[l]_-\simeq\ar[r]^-\simeq&\Otimes[Q]{F}{k}\ar[d]
\\
S\ar[r]&K^S&\Otimes[Q']{S}{K^{Q'}}\ar[l]_-\cong
&\Otimes[Q']{F'}{K^{Q'}}\ar[l]_-\simeq\ar[r]^-\simeq
&\Otimes[Q']{F'}{l}
}\end{split}\end{equation}
Using the properties catalogued above, it is straightforward to show that each square in this diagram is a pushout square. 

Since $M$ is $\pd$-test over $R$, Lemma~\ref{lem190915a} implies that $\Lotimes{K^R}M$ is $\pd$-test over $K^R$. 
Since the other arrows in the top row of the above diagram are equivalences, it is straightforward to show that
$\Lotimes{K^R}M$ is $\pd$-test over $\Otimes[Q]F{K^Q}$ by restriction of scalars, and 
$\Lotimes[\protect{\Otimes[Q]F{K^Q}}]{(\Otimes[Q]Fk)}{(\Lotimes{K^R}M)}$ is $\pd$-test over $\Otimes[Q]Fk$ by base-change.

From our assumptions, it is straightforward to show that $\Otimes[Q']{F'}{l}\cong\Otimes[k]{l}{(\Otimes[Q]{F}{k})}$.
Thus, Theorem~\ref{thm190913a} implies that 
the following is $\pd$-test over $\Otimes[Q']{F'}l$.
\begin{multline*}
\Lotimes[\protect{\Otimes[Q]Fk}]{(\Otimes[Q']{F'}{l})}{(\Lotimes[\protect{\Otimes[Q]F{K^Q}}]{(\Otimes[Q]Fk)}{(\Lotimes{K^R}M)})}\\
\simeq
\Lotimes[\protect{\Otimes[Q']{F'}{K^{Q'}}}]{(\Otimes[Q']{F'}{l})}{(\Lotimes[\protect{\Otimes[Q]F{K^Q}}]{(\Otimes[Q']{F'}{K^{Q'}})}{(\Lotimes{K^R}M)})}
\end{multline*}
From the bottom-right equivalence in the diagram~\eqref{eq191118a}, we conclude that the DG module
$\Lotimes[\protect{\Otimes[Q]F{K^Q}}]{(\Otimes[Q']{F'}{K^{Q'}})}{(\Lotimes{K^R}M)}$ is $\pd$-test over $\Otimes[Q']{F'}{K^{Q'}}$
by restriction of scalars.
From the other equivalences in the bottom row of the diagram, we see that
the following DG module is $\pd$-test over $K^S$;
\begin{align*}
\Lotimes[\protect{\Otimes[Q]F{K^Q}}]{(\Otimes[Q]Fk)}{(\Lotimes{K^R}M)}
&\simeq\Lotimes[K^R]{K^S}{(\Lotimes{K^R}M)}
\simeq\Lotimes[S]{K^S}{(\Lotimes{S}M)}
\end{align*}
the first isomorphism comes from the pushout property of the diagram.
Lemma~\ref{lem190915a} implies that $\Lotimes{S}M$ is $\pd$-test over $S$. 
This concludes the proof in Case~1.

Case 2: $\m S=\n$. In this case, consider the following natural diagram of flat local ring homomorphisms.
$$\xymatrix{
R\ar[r]\ar[d]&S\ar[d]\\
\comp R\ar[r]&\comp S}$$
Then $\Lotimes{\comp R}M$ is $\pd$-test over $\comp R$ by Theorem~\ref{thm140825a}.
Case~1 then implies that the following complex is $\pd$-test over $\comp S$.
\begin{align*}
\Lotimes[\comp R]{\comp S}{(\Lotimes{\comp R}M)}
&\simeq\Lotimes[S]{\comp S}{(\Lotimes{S}M)}
\end{align*}
Another application of Theorem~\ref{thm140825a} implies that $\Lotimes{S}M$ is $\pd$-test over $S$.

Case 3: the general case
follows from Case~2 as in the proof of Theorem~\ref{thm140825a}.
\end{proof}

\begin{disc}\label{disc190915a}
It would be nice to answer~\cite[Question~3.7]{celikbas:tgp} in general.
This is equivalent to proving that the conclusion of Theorem~\ref{thm140915a} holds without the assumption that $k\to l$ is algebraic.
Using a transcendence basis, because of Theorem~\ref{thm140915a} and its proof, this reduces to 
proving that the converse to Theorem~\ref{thm190913a}\eqref{thm190913a1} holds in the case where $k\to l$ is purely transcendental
of finite transcendence degree.
Unfortunately, we do not know how to prove this even when $k\to l$ is purely transcendental of transcendence degree 1.
\end{disc}

\section{G-Dim-Test Results} \label{sec190913a}

Throughout this section,
$(A,\m_A,k)$ is a weakly local DG algebra.

\

Here, we analyze ascent properties for $\gcdim$- and $\gdim$-test modules using the techniques of the preceding section.
Our main result here is Theorem~\ref{thm190915b}.

We say that $C\in\catdfb(A)$ is a \emph{semidualizing} DG $A$-module if the natural homothety morphism
$\chi^A_C\colon A\to\Rhom[A]CC$ is an isomorphism in $\catd(A)$.
For instance, since $A$ is homologically bounded, $A$ itself is a semidualizing DG $A$-module.

Assume that $C$ is a semidualizing DG $A$-module.
We say that $N\in\catdfb(A)$ is \emph{derived $C$-reflexive} over $A$ if 
$\Rhom[A]NC\in\catdb(A)$ and the natural biduality morphism
$\delta^N_C\colon N\to\Rhom[A]{\Rhom[A]NC}C$ is an isomorphism in $\catd(A)$.
Then $M\in\catdfb(A)$ is \emph{$\gcdim$-test} over $A$ if the following condition holds for all  $N\in\catdfb(A)$:
If $\Lotimes[A] MN\in\catdb(A)$, then $N$ is derived $C$-reflexive.
In the case $C=A$, we write $\gdim$-test instead of $\gkdim{A}$-test.
Fact~\ref{fact190913a} implies that $k$ is $\gcdim$-test over $A$ for each semidualizing DG $A$-module $C$
because every perfect DG $A$-module is derived $C$-reflexive.
In particular, $k$ is $\gdim$-test.

We proceed as in the preceding section with some preparatory lemmas.

\begin{lem}\label{lem190929a}
Assume that $A$ is a finite-dimensional DG $k$-algebra with $A_0=k$ and $\HH_0(A)\neq 0$.
Let $k\to l$ be a  field extension, set $B=\Otimes[k]lA$, and consider the natural morphism of DG algebras $A\to B$.
Let $C\in\catdfb(A)$ be given.
Then $\Lotimes[A]BC$ is semidualizing over $B$ if and only if $C$ is semidualizing over $A$.
\end{lem}

\begin{proof}
Note that since $l$ is free over $k$, we have $B\simeq\Lotimes[k]lA$.
Furthermore, the $A$-algebra $B=\Otimes[k]lA$ is free with semi-basis
concentrated in degree 0. These facts justify three of the unspecified isomorphisms in the following diagram, while the fourth one
is tensor-cancellation.
$$\xymatrix@R=6mm{
\Lotimes[k]lA\ar[r]^-{\Lotimes[k]l{\chi^A_C}}\ar[d]_-\simeq&\Lotimes[k]l{\Rhom[A]CC}\ar[d]^-\simeq\\
B\ar[d]_-{\chi^B_{\Lotimes[A]BC}}&\Lotimes[A]{(\Lotimes[k]lA)}{\Rhom[A]CC}\ar[d]^-\simeq\\
\Rhom[B]{\Lotimes[A]BC}{\Lotimes[A]BC}&\Lotimes[A]B{\Rhom[A]CC}\ar[l]_-\simeq
}$$
It follows that $\chi^B_{\Lotimes[A]BC}$ is an isomorphism if and only if $\Lotimes[k]l{\chi^A_C}$ is an isomorphism;
since $l$ is free and non-zero over $k$, the morphism $\Lotimes[k]l{\chi^A_C}$ is an isomorphism is and only if $\chi^A_C$ is one.
Since $\Lotimes[A]BC\in\catdfb(B)$, the desired conclusion now follows. 
\end{proof}

\begin{lem}\label{lem190929b}
Let $A$ be a finite-dimensional DG $k$-algebra with $A_0=k$ and $\HH_0(A)\neq 0$.
Let $k\to l$ be a  field extension, set $B=\Otimes[k]lA$, and consider the natural morphism of DG algebras $A\to B$.
Let $M\in\catdfb(A)$ be given, and let $C$ be a semidualizing DG $A$-module.
Then $\Lotimes[A]BM$ is derived $(\Lotimes[A]BC)$-reflexive over $B$ if and only if $M$ is derived $C$-reflexive over $A$.
\end{lem}

\begin{proof}
As in the  proof of Lemma~\ref{lem190929a}, in $\catd(A)$ we have
$\Rhom[B]{\Lotimes[A]BM}{\Lotimes[A]BC}\simeq\Lotimes[k]l{\Rhom[A]MC}$
and therefore $\Rhom[B]{\Lotimes[A]BM}{\Lotimes[A]BC}\in\catdb(B)$ if and only if $\Rhom[A]MC\in\catdb(A)$.
Assume therefore that $\Rhom[A]MC\in\catdb(A)$, and
use the next diagram
$$\xymatrix@R=6mm{
\Lotimes[k]lM\ar[r]^-{\Lotimes[k]l{\delta^M_C}}\ar[dd]_-\simeq&\Lotimes[k]l{\Rhom[A]{\Rhom[A]MC}C}\ar[d]^-\simeq\\
&\Lotimes[A]{(\Lotimes[k]lA)}{\Rhom[A]{\Rhom[A]MC}C}\ar[d]^-\simeq\\
\Lotimes[A]{(\Lotimes[k]lA)}M\ar[dd]_-\simeq&\Lotimes[A]B{\Rhom[A]{\Rhom[A]MC}C}\ar[d]^-\simeq\\
&\Rhom[B]{\Lotimes[A]B{\Rhom[A]MC}}{\Lotimes[A]BC}
\\
\Lotimes[A]BM\ar[r]^-{\delta^{\Lotimes[A]BM}_{\Lotimes[A]BC}}&\Rhom[B]{\Rhom[B]{\Lotimes[A]BM}{\Lotimes[A]BC}}{\Lotimes[A]BC}\ar[u]_-\simeq
}$$
as in the  proof of Lemma~\ref{lem190929a} to complete the argument.
\end{proof}

Here are the versions of Theorem~\ref{thm190913a} and Lemma~\ref{lem190915a} for this context.

\begin{thm}\label{thm190913a'}
Let $A$ be a finite-dimensional DG $k$-algebra with $A_0=k$ and $\HH_0(A)\neq 0$.
Let $k\to l$ be a  field extension, set $B=\Otimes[k]lA$, and consider the natural morphism of DG algebras $A\to B$.
Let $M\in\catdfb(A)$ be given, and let $C$ be a semidualizing DG $A$-module.
\begin{enumerate}[\rm(a)]
\item \label{thm190913a'1}
If 
$\Lotimes[A] BM$ is $\gkdim{\Lotimes[A]BC}$-test  over $B$, then 
$M$ is  $\gkdim{C}$-test  over $A$.
\item \label{thm190913a'2}
The converse of part~\eqref{thm190913a'1} holds if the extension $k\to l$ is algebraic.
\item \label{thm190913a'11}
If 
$\Lotimes[A] BM$ is $\gdim$-test  over $B$, then 
$M$ is  $\gdim$-test  over $A$.
\item \label{thm190913a'22}
The converse of part~\eqref{thm190913a'1} holds if the extension $k\to l$ is algebraic.
\end{enumerate}
\end{thm}

\begin{proof}
\eqref{thm190913a'1}
Argue as in the proof of Theorem~\ref{thm190913a}\eqref{thm190913a1}, using Lemma~\ref{lem190929b}.

\eqref{thm190913a'2}
Assume now that $M$ is a  $\gcdim$-test complex over $A$.

Case 1: $k\to l$ is finite.
To show that $\Lotimes[A] {B}M$ is $\gkdim{\Lotimes[A]BC}$-test over $B$, 
let $N\in\catdfb(B)$ be such that
$\Lotimes[B] {(\Lotimes[A] {B}M)}N\in\catdb(B)$.
Since $k\to l$ is finite, we have $B\in\catdfb(A)$, so 
$N\in\catdfb(A)$.
Moreover, 
$\Lotimes[A]MN\simeq\Lotimes[B] {(\Lotimes[A] {B}M)}N\in\catdb(A)$.
As $M$ is $\gcdim$-test  over $A$, the DG module $N$ is derived $C$-reflexive over $A$.
To conclude that $N$ is derived $(\Lotimes[A]BC)$-reflexive over $B$, we use the Auslander class.

Since $A$ is finite dimensional over $k$, it has a dualizing DG module $D^A$. 
From the definition of $B$ as $\Otimes[k]lA$, it is straightforward to show that
$D^B=\Lotimes[A]B{D^A}$ is a dualizing DG $B$-module. 
As in~\cite[(4.7)~Theorem]{christensen:scatac} or~\cite[(4.1.7)]{avramov:rhafgd} one shows that
$N$ is derived reflexive over $A$ if and only if $N$ is in the Auslander class $\cata_{\Rhom[A]C{D^A}}(A)$,
and similarly over $B$.
Moreover, arguing as in~\cite[(5.3)~Proposition]{christensen:scatac}, one sees that
$N$ is in $\cata_{\Rhom[A]C{D^A}}(A)$ if and only if it is in $\cata_{\Lotimes[A]B{\Rhom[A]C{D^A}}}(B)=\cata_{\Rhom[B]{\Lotimes[A]BC}{D^B}}(B)$, that is,
$N$ is derived $C$-reflexive over $A$ if and only if it is derived $(\Lotimes[A]BC)$-reflexive over $B$.
This concludes the proof in Case 1.

Case 2: $k\to l$ is algebraic.
Argue as in Case 2 of the proof of Theorem~\ref{thm190913a}\eqref{thm190913a2}.

\eqref{thm190913a'11}--\eqref{thm190913a'22}
These are the special case $C=A$ of parts~\eqref{thm190913a'1}--\eqref{thm190913a'2}.
\end{proof}

As in Section~\ref{sec190908c}, we use the exterior DG algebra structure on the Koszul complex to prove our ascent result for $\gdim$-test complexes.

\begin{fact}\label{fact190929a}
Let $K=K^R(\x)$ be the Koszul complex on $\x=x_1,\ldots,x_n\in\m$, and let $C\in\catdfb(R)$.
Then $\Lotimes KC$ is semidualizing over $K$ if and only if $C$ is semidualizing over $R$
by~\cite[Lemma~A.3(a)]{christensen:dvke}.
\end{fact}

\begin{lem}\label{lem190915a'}
Let $K=K^R(\x)$ be the Koszul complex on $\x=x_1,\ldots,x_n\in\m$.
Let $M\in\catdfb(R)$ be given, and let $C$ be a semidualizing $R$-complex.
\begin{enumerate}[\rm(a)]
\item \label{lem190915a'1}
$\Lotimes KM$ is $\gkdim{\Lotimes KC}$-test over $K$ if and only if $M$ is $\gcdim$-test over $R$.
\item \label{lem190915a'2}
$\Lotimes KM$ is $\gdim$-test over $K$ if and only if $M$ is $\gdim$-test over $R$.
\end{enumerate}
\end{lem}

\begin{proof}
Again, we focus on part~\eqref{lem190915a'1}.
For one implication, argue as in the descent result~\cite[Theorem~3.2]{celikbas:tgp}.
For the converse, assume  that $M$ is  $\gcdim$-test  over $R$.

Case 1: $R$ has a dualizing complex $D^R$.
To show that $\Lotimes{K}M$ is $\gkdim{\Lotimes KC}$-test over $K$, 
let $N\in\catdfb(K)$ be such that
$\Lotimes[K] {(\Lotimes{K}M)}N\in\catdb(K)$.
Since $K\in\catdfb(R)$, we have 
$N\in\catdfb(R)$.
Moreover, 
$\Lotimes MN\simeq\Lotimes[K] {(\Lotimes{K}M)}N\in\catdb(R)$.
As $M$ is $\gcdim$-test  over $R$, the DG module $N$ is derived $C$-reflexive over $R$,
Now argue as in the proof of Theorem~\ref{thm190913a'} using $D^R$ and $D^K=\Lotimes K{D^R}$ to prove that
$N$ is derived $\Lotimes KC$-reflexive over $K$.

Case 2: $(\x)R$ is $\m$-primary.
Proposition~\ref{thm140825ax} implies that $\Lotimes{\comp R}M$ is $\gkdim{\Lotimes{\comp R}C}$-test for $\comp R$.
Set $\comp K=\Lotimes{\comp R}K$.
Recall that $\comp R$ has a dualizing complex,
so Case 1 implies that the DG $\comp K$-module
$$\Lotimes[\comp R]{\comp K}{(\Lotimes{\comp R}M)}\simeq\Lotimes{\comp K}M\simeq\Lotimes KM$$
is $\gkdim{\Lotimes{\comp K}C}$-test for $\comp K$;
the second isomorphism here uses the assumption on $(\x)R$, which implies that
the natural morphism $K\to\comp K$ is a quasiisomorphism of DG algebras.
This quasiisomorphism furthermore implies that
this DG module
is $\gkdim{\Lotimes KC}$-test for $K$,
by restriction of scalars.

Case 3: the general case. 
Let $\y=y_1,\ldots,y_m\in\m$ be a generating sequence for $\m$.
By Case 2, the DG $K^R(\x,\y)$-module $\Lotimes{K^R(\x,\y)}M$ is $\gkdim{\Lotimes{K^R(\x,\y)}C}$-test over $K^R(\x,\y)$.
Argue as in the descent implication, from the DG algebra morphism $K\to K^R(\x,\y)$ to conclude that 
$\Lotimes KM$ is $\gkdim{\Lotimes KC}$ test over $K$, as desired.
\end{proof}

\begin{thm}\label{thm140915a'}
Assume that $\vf$ is flat with regular closed fibre such that the induced field extension $k \to l$ is algebraic.
Let $M\in\catdfb(R)$ be given,
and let $C$ be a semidualizing $R$-complex.
\begin{enumerate}[\rm(a)]
\item \label{thm140915a'1}
 $M$ is  $\gcdim$-test  over $R$ if and only if $\Lotimes SM$ is  $\gkdim{\Lotimes SM}$-test  over $S$.
\item \label{thm140915a'2}
 $M$ is   $\gdim$-test  over $R$ if and only if $\Lotimes SM$ is  $\gdim$-test  over $S$.
\end{enumerate}
\end{thm}

\begin{proof}
Argue as in the proof of Theorem~\ref{thm140915a}.
\end{proof}

We conclude with the main result of this section.

\begin{thm}\label{thm190915b}
Assume that $\vf$ is Gorenstein such that the induced field extension $k \to l$ is algebraic.
Let $M\in\catdfb(R)$,
and let $C$ be a semidualizing $R$-complex.
\begin{enumerate}[\rm(a)]
\item \label{thm190915b1}
$M$ is   $\gcdim$-test  over $R$ if and only if $\Lotimes SM$ is  $\gkdim{\Lotimes SC}$-test  over $S$.
\item \label{thm190915b2}
$M$ is   $\gdim$-test  over $R$ if and only if $\Lotimes SM$ is  $\gdim$-test  over $S$.
\end{enumerate}
\end{thm}

\begin{proof}
Again we focus on part~\eqref{thm190915b1}.
One implication is from Proposition~\ref{prop140825a}\eqref{prop140825a2}.
For the converse, assume that $M$ is  $\gcdim$-test  over $R$.
Let $R\to R'\to \comp S$ be a Cohen factorization of the semicompletion $\grave\vf\colon R\to\comp S$,
This provides isomorphisms
\begin{gather*}
\Lotimes[R']{\comp S}{(\Lotimes{R'}C)}\simeq\Lotimes {\comp S}C\simeq\Lotimes[S]{\comp S}{(\Lotimes{S}C)}\\
\Lotimes[R']{\comp S}{(\Lotimes{R'}M)}\simeq\Lotimes {\comp S}M\simeq\Lotimes[S]{\comp S}{(\Lotimes{S}M)}.
\end{gather*}
Theorem~\ref{thm140915a'} implies that $\Lotimes{R'}M$ is $\gkdim{\Lotimes{R'}C}$-test over $R'$.
Because of the displayed isomorphisms,
Proposition~\ref{thm140825ax} implies that $\Lotimes {\comp S}M$
is $\gkdim{\Lotimes {\comp S}C}$-test over $\comp S$.
Thus, 
$\Lotimes SM$ is $\gkdim{\Lotimes SC}$-test over $S$ by Proposition~\ref{prop140825a}\eqref{prop140825a2}.
\end{proof}

\section*{Acknowledgments}

I am grateful to 
Olgur Celikbas,
Craig Huneke,
Srikanth Iyengar,
Mohsen Gheibi, and
Justin Lyle
for useful conversations about this work. 
Much of this work was completed around the conference Morgantown Algebra Days (MAD) 2019.
I appreciate the fact that the organizers of MAD 2019  created a productive workshop environment. 
Lastly, I am thankful for all the hard work Roger and Sylvia Wiegand have done over their careers to create/discover
inspiring mathematics and to build a welcoming and inclusive environment in the commutative algebra community.
They are two of the kindest people I know.

\providecommand{\bysame}{\leavevmode\hbox to3em{\hrulefill}\thinspace}
\providecommand{\MR}{\relax\ifhmode\unskip\space\fi MR }
\providecommand{\MRhref}[2]{%
  \href{http://www.ams.org/mathscinet-getitem?mr=#1}{#2}
}
\providecommand{\href}[2]{#2}

\end{document}